\documentclass[a4paper,12pt]{amsart}
\usepackage{lipsum}
\usepackage{listings,xcolor}
\usepackage[affil-it]{authblk}
\usepackage{graphicx}
\usepackage{listings}
\usepackage{amssymb,amsfonts}
\usepackage{amsmath}
\usepackage[all,arc]{xy}
\usepackage{enumerate}
\usepackage{mathrsfs}
\usepackage[latin1]{inputenc}
\usepackage{palatino}
\usepackage{color}
\usepackage{graphicx}
\usepackage{graphics}
\usepackage{tikz}
\usepackage{epsfig}
\usepackage{url}
\usepackage{tikz}
\usepackage{mathtools}
\usepackage{epsfig}
\usepackage{listings,xcolor}
\usepackage{listings}
\usepackage{hyperref}
\usepackage[affil-it]{authblk}
\usetikzlibrary{graphs}
\usetikzlibrary{matrix}
\usetikzlibrary{arrows}
\usetikzlibrary{chains}

\theoremstyle{definition}

\newtheorem{thm}{Theorem}[section]

\newtheorem{lem}[thm]{Lemma}

\theoremstyle{definition}
\newtheorem{defn}[thm]{Definition}

\theoremstyle{remark}
\newtheorem{rem}[thm]{Remark}

\theoremstyle{remark}

\numberwithin{equation}{section}
\makeatletter
\g@addto@macro{\endabstract}{\@setabstract}
\newcommand{\authorfootnotes}{\renewcommand\thefootnote{\@fnsymbol\c@footnote}}%
\makeatother

\definecolor{checkcolor}{rgb}{0.95, 0.95, 0.95}
\newsavebox{\definitionbox}
		{\end{minipage}\end{lrbox}%
	\begin{center}{\colorbox{checkcolor}{\usebox{\definitionbox}}}%
	\end{center}}
\begin{document}
	\begin{center}
		\LARGE
		Bergman's Centralizer Theorem and quantization \par \bigskip
		
		\normalsize
		\authorfootnotes
		Alexei Kanel Belov\footnote{kanelster@gmail.com}
		, Farrokh Razavinia \footnote{f.razavinia@phystech.edu}
		Wenchao Zhang\footnote{whzecomjm@gmail.com}
		
		 Moscow Institute of Physics and Technology \par
	 Bar-Ilan Univeristy \par \bigskip
	
		
	\footnote{	This research was funded by the Russian Science Foundation (RNF) educational scholarship with number (\textit{No}17-11-01377). }\textsuperscript{1}
		\footnote{2010 Math Subject Class: Primary 16R50, 13A99  . ; Secondary 	53D55 .}
		\footnote{keywords: quantization, Kontsevich, Bergman theorem, generic matrices.}
	\end{center}
	
	\begin{abstract}
		We prove Bergman's theorem \cite{Berg69} on centralizers by using generic matrices and Kontsevich's quantization method. For any field
        $\textbf{k} $ of positive characteristics,  set $A=\textbf{k} \langle x_1,\dots,x_s\rangle$ be a free associative algebra, then any centralizer $\mathcal{C}(f)$ of nontrivial element $f\in A\hspace*{-0.1cm}\setminus \hspace*{-0.1cm} \textbf{k} \hspace*{ 0.1cm}$ is a ring of polynomials on a single variable. We also prove that there is no commutative subalgebra with transcendent degree $\geq 2$ of $A$.
	\end{abstract}
	
\section{Introdution}
Quantization ideas provides a new vision in some classical areas in mathematics including polynomial automorphisms and Jacobian conjecture. In the papers \cite{Belov05,Belov07} , Alexei Kanel Belov and Maxim Kontsevich used antiquantizations for to obtain a proof for to show the equivalence of the Jacobian conjecture and the Dixmier conjecture. And P. S. Kolesnikov\cite{Kole00,Kole01} reproved the famous Makar-Limanov theorem on the construction of algebraically closed skew fields by using formal quantization by using the formal quantization. 

In the series of papers\cite{Art78,Art84,Art91,Art98,Art08},  V. A. Artamonov discovered relations of metabelian algebras and the Serre conjecture with quantization. Some of Artamonov's ideas had been used by U. Umirbaev\\ \cite{Umir12} for Jacobian conjecture in the case of metabelian algebras. And in 2007, Umirbaev in his classical paper \cite{Umir07}, used metabelian algebras in order to prove the famous Anick conjecture (it was second ingredient, the main ingredient was the famous theory of U.Umirbaev and I.Shestakov).

As to review some classical algebra and algebraic geometry, especially polynomial automorphisms from quantization point of view seems to be important for us, so in this paper; which we have devoted to the Bergman's theorem on centralizers of free associative algebras via quantization; we obtain the following theorem with somehow more general results:

\begin{thm}
	Suppose $\mathrm{\textit{tr.deg}}\textbf{k}\langle A,B\rangle=2$, for $\textbf{k}$ is a field with positive characteristics and $A, B$ are matrices over $\textbf{k}$. 
	
	Consider an arbitrary lifting $A$ and $B$ in $\mathrm{Mat}_n(\textbf{k}\langle x_1,\dots,x_s\rangle[[\mathfrak{h}]], *)$, then $\hat{A}$ and $\hat{B}$ do not commute. Moreover,
	\begin{equation}
	\frac{1}{\mathfrak{h}}[\hat{A},\hat{B}]_*=\begin{pmatrix} \frac{1}{\mathfrak{h}}\{\lambda_1,\mu_1\}& & 0 \\
	& \ddots & \\
	0 & &\frac{1}{\mathfrak{h}}\{\lambda_n,\mu_n\}\end{pmatrix}
	\end{equation}
	in the basis such that $A$ and $B$ can be upper diagonalized.
	
	Where for $i \in \{1, \dots , n \}$, $\lambda_i$  and $\mu_i$ are respectively the eigenvalues of $A$ and $B$ which will make the diagonal elements of the quantized $\hat{A}$ and $\hat{B}$. 
	
\end{thm}

\begin{rem}
	In theorem 1.1, $\mathrm{\textit{tr.deg}} \textbf{k} \langle A,B \rangle =2$ and the independency of $\lambda_i$'s and $\mu_i$'s for some $i$ will ensure that the Poisson bracket of $\lambda_i$ and $\mu_i$ is non zero. (In fact because the characteristic polynomial of $A$ and $B$ are both irreducible, so that's true for all $i$.) 
\end{rem}

But first of all let us to mention the famous Bergman's theorem:
\begin{thm}[The Bergman's centralizer theorem]
	The maximum commutative subalgebra of a free associative algebra $A$ is isomorphic to the ring of polynomials in one variable.
\end{thm}
This theorem already has been proved in \cite{Berg69}, but we have to reprove it because of its importance for our ideas. And the main reason for to reprove it is to find its relation with quantization in Kontsevich sense.\\

The proof of the theorem 1.3 can be reduced by using the following steps:

\vspace*{0.18cm}
{\bf 1.} Fist suppose we have a commuting, then we can make a reduction to the generic matrices. And when we make this reduction then we will have two commuting generic matrices.

\begin{rem}
	And in this case by using the fact that when we come to the quantization of the generic matrices then those matrices are not allowed just to commute but to have no other relations, because otherwise we can consider the algebra generic over this ring with our generators as a generic matrix for every lifting!
\end{rem}

\vspace*{0.18cm}
{\bf 2.} Let $\mathbf R$ be the ring of polynomials and $A$ and $B$ matrices over $\mathbf R$ and suppose we have generators $X_{ij}$s and $Y_{ij}$s and  a polynomial over generic matrices.

 Now if we consider not polynomials over generic matrices but instead of $X_{ij}$ and $Y_{ij}$ we replace them by quantizations, i.e. not with elements from polynomial ring $R$ but with elements from quantized polynomial ring $\hat{R}$, then we will get two matrices which cannot commute if they have relation. And if they do not have relation apart of commuting, because the theorem 1.1 is not allowed because we work on a ring of generic matrices and if they in a ring of generic matrices commute without any relations then when we come to the quantization then we loose any relation of them and then we can consider instead of our free generators of free associative algebra, its homomorphisms on the algebra of generic matrices to the algebra over matrices over new quantization.    

\vspace*{0.18cm}
{\bf 3.} As it comes out by the definition of a generic matrix ring of size two which is a ring generated by generic matrices of order 2 and set\\

\hspace*{2.4cm}$X^1 = \begin{pmatrix} X_{11}^{1} & X_{12}^{1} \\ X_{21}^{1} &  X_{22}^{1} \end{pmatrix}$, $\dots$ , $X^s = \begin{pmatrix} X_{11}^{s} & X_{12}^{s} \\ X_{21}^{s} &  X_{22}^{s} \end{pmatrix};$

\vspace*{0.14cm}
such that the algebra generated by $X^s$'s is the algebra of generic matrices of order 2 and this is a subalgebra of matrix algebra of the ring of polynomials of $4s$ variables. So it is the algebra of generic matrices over the ring of polynomials.

\begin{rem}
	If we have free associative algebra over $a_1,\dots,a_s$, then we will have homomorphisms $\varphi$'s to $X^1,\dots,X^s$ from $a_i$s and it means that if $f$ and $g$ commute then $\varphi(f)$ and $\varphi(g)$ also commute.
\end{rem}

\vspace*{0.18cm}
{\bf 4.} Now let us put $X_{ij}^{l}$ instead of $X^i$'s and let us define $R := \textbf{k}[X_{ij}^{l}]$ and $\hat{R}$; for $\hat{R}$ to be the quantization of $R$. And this means that
if we have free associative polynomial algebras $f(a_1,\dots,a_s)$ and $g(a_1,\dots,a_s)$ over $a_1,\dots,a_s$ according to the remark 1.5, then we will get in matrix algebras $\hat{f}(\hat{X}_{1},\dots,\hat{X}_s)$ and $ \hat{g}(\hat{X}_{1},\dots,\hat{X}_s)$.

In other words, if we suppose the associative algebra $\textbf{k}\langle X_1, \dots,X_s\rangle$ then from the other hand we will have $\textbf{K}\langle \hat{M}_1,\dots,\hat{M}_s \rangle$ for $\hat{M}^l = (\hat{X}_{ij}^{l})$ for $\hat{X}_{ij}^{l}$  from the ring of quantization of the ring $R$ i.e. $\hat{R}$ and with modular $\mathfrak{h}$ will be matrices over $X_{ij}$.

\vspace*{0.18cm}
{\bf 5.} Now suppose that all the previous steps are satisfying for any $n$ relations in generic matrices.

\begin{rem}
	 Here relations are some polynomials $P_n$. 
\end{rem}

\vspace*{0.18cm}
{\bf 6.} Now suppose for any $n$ and for any $f ,g$ as before, we have $P_n(f , g) = 0$.

 Then for any $n_1 > n_2$ we have that $P_{n_1}$ will contain  $n_2 \times n_2$ non generic matrices; and because it is known that when it does not have zero divisors, it means that all of them are irreducible and there is minimal relations between generic matrices and then if $n_1 > n_2$ then the relation $P_{n_1}$ is divisible by $P_{n_2}$ and hence because of irreducibility they must coincide on some step which means that they have relation for all $n$'s and this provides that they do not have relation at all!
 
  Because if $n$ is big then all the identities of $n \times n $ matrices and hence minimum identity has degree $\frac{n}{2}$ and then the degree of this polynomial if we substitute instead of $f$ and $g$ times maximum degree of $f$ and $g$, so if we consider two big matrices then we do not have relations and it means that it is impossible that in a free associative algebra we do not have relation but in any size of matrices we have relation! 

A contradiction!\\

And this is the situation which occurs in quantization.	

\vspace*{0.38cm}
{\bf 7.} And then we need to prove that for $n \times n$ polynomials we cannot have any relations but just free commutative algebra of transcendent degree 2 and then Theorem 1.1 does not allow us to have this, because of the famous Bergman theorem about centralizers!

\vspace*{0.18cm}
{\bf 8.} Now in this step let us look at it from an another point of view:

Consider the quantization ring $\hat{R}$ and suppose for all $i,j$ we have $\hat{X}_{ij}^{l} \cong X_{ij}^{l}$ mod $\mathfrak{h}$ in $\hat{R}$, then two matrices $\hat{A}$ and $\hat{B}$ does not commute anymore!

\begin{rem}
	This is enough for Bergman's theorem.
\end{rem}

Now let's see how theorem 1.1 implies theorem 1.3.

\vspace*{0.18cm}
{\bf 9.}  We want to prove that two elements cannot commute.

Suppose we have a free associative algebra $\textbf{k}\langle X_1,\dots,X_s \rangle$ such that $f,g \in \textbf{k}\langle X_1,\dots,X_s\rangle$ and suppose $\textbf{k}\langle f,g \rangle \simeq \textbf{k}[y,z]$, i.e. they commute but are algebraically independent.

\vspace*{0.18cm}
{\bf 10.} And now if we consider the reduction to the generic matrices: 

Then we have $\textbf{k}   \langle   A_{1}^{n_i},\dots,A_{s}^{n_i}   \rangle $ for $\hat{f}_{n_i},\hat{g}_{n_i} \in \textbf{k} \langle A_{1}^{n_i},\dots, A_{s}^{n_i} \rangle $. And hence $\textbf{k} \langle  X_1,\dots,X_s \rangle $ naturally maps in $\textbf{k} \langle A_{1}^{n_i},\dots,A_{s}^{n_i} \rangle $ and also $f,g$ maps to $\hat{f}_{n_i} , \hat{g}_{n_i}$ respectively and because we had $[f_{n_i},g_{n_i}] = 0$ so we will have $[\hat{f}_{n_i},\hat{g}_{n_i}] =0.$

\vspace*{0.18cm}
{\bf 11.} Now suppose there is relation $R_{n_i}$ such that $R_{n_i} (f_{n_i },g_{n_i}) = 0 $ for each $i$. And because the algebra of generic matrices $M_n$ is a domain, then $R_{n_i}$ is irreducible. And so for $n_j > n_i$ if $R_{n_i} (f_{n_i} , g_{n_i}) = 0 $ then we have $R_{n_j} (f_{n_j},g_{n_j}) = 0 $ and hence $R_{n_j}$ is divisible by $R_{n_i}$ and so we will have that $R_{n_j} = R_{n_i}$.

\vspace*{0.18cm}
{\bf 12.} Now consider an arbitrary relation $R$ such that $\text{deg} R(f,g) = q$ for $n_i > q,$ then $R(f,g) = 0$ in associative algebra, otherwise it is identity because of Levi's lemma (minimal degree of identity).\cite{Amit50}

\hspace*{-0.1cm}And we note that the target is to prove that for $f,g$ and hence for big $n_i$, $\textbf{k}[f^{n_i},g^{n_i}] \simeq \textbf{k}[x,y]$ and this means that for big $n_i$, $f^{n_i}$ and $g^{n_i}$ does not produce new relations.
And hence if we suppose $\textbf{k}[f,g] \simeq \textbf{k}[x,y]$, then for big $n$, $\textbf{k}[f^n,g^n]$ is still isomorphic to $\textbf{k}[x,y]$.

\vspace*{0.18cm}
{\bf 14.} Now we want to prove that if $\textbf{k}[ f,g] \simeq \textbf{k}[x,y]$ then it is also correct for $X_{1}^{n},\cdots,X_{s}^{n}$ the generators of the algebra of generic matrices $(X_{ij}^{l})$ for $l \in \{ 1,\dots,s\}$ and $i,j \in \{ 1,\dots,n \}$.

\vspace*{0.18cm}
{\bf 15.} And now if we replace $X^l$ by $X$ and $\hat{f}, \hat{g}$ by $f, g$ respectively, then from one hand we have that $\hat{f}$ and $\hat{g}$ commute because in our free associative algebra we had that $f$ and $g$ have been commuting and from the other hand we shall prove that they don't commute because of theorem 1.1! And this will provide us again a contradiction!

\vspace*{0.18cm}
{\bf 16.} We proved that if $f$ and $g$ commute then $\mathrm{\textit{tr.deg}} \textbf{k} \langle f,g\rangle = 1.$

Then one can prove according to the Bergman's centralizer theorem that for some $\mathfrak{h}$, $f = \varphi(\mathfrak{h})$ and $g = \psi(\mathfrak{h})$ for some polynomials $\varphi$ and $\psi$.

\begin{rem}
	Here again let us mention that our first goal in this article is to show that theorem 1.1 implies the Bergman's theorem on centralizers and then to conclude and to prove the direct relation between the Bergman's theorem and the Kontsevich quantization in a classic way!
\end{rem}

 \begin{rem}
 	This proof is important because it shows that staff with quantization are related to some classical results of 70's from Bergman's results on commutative algebras!
 \end{rem}

\section{Algebra of Generic matrices}
Let $A= \textbf{k}\langle x_1,\dots,x_s\rangle$ be a free associative algebra over the ground field $\textbf{k}$ on the variables $x_1,x_2,\dots,x_s$.

\begin{defn}
	A \textit{generic matrix} is a matrix whose entries are distinct commutative indeterminates, and the so called algebra of generic matrices of order $m$ is generated by associative generic $m \times m$ matrices.
\end{defn}
 
 In order to prove the main results we recall some useful stuffs about generic matrices.
 
  Consider the ring of generic matrices of the free associative algebra $A$, Let $X_1,\dots,X_s$ be $n\times n$ matrices with entries $x_{ij}^{v}$ which are independent central variables. The subring of the matrix algebra $M_n(\textbf{k}[x_{ij}^{v}])$ generated by the matrices $X_v$ is called \textbf{the algebra of generic matrices}\cite{Artin99}, which we denote it by $R=\textbf{k}\langle X_1,\dots,X_s\rangle$.  

The algebra of generic matrices is prime, and every prime, relatively free, finitely generated associative $PI-$algebra is isomorphic to an algebra of generic matrices. And if we include taking traces as an operator in the signature, then we get the algebra of generic matrices with trace.

\begin{defn}
	A \textit{T-ideal} is a completely characteristic ideal, i.e., stable under any endomorphism.
\end{defn}

We note that there is a natural homomorphism from the free associative algebra to this algebra of generic matrix with some kernel $T$-ideal:

\begin{equation}
\pi: \textbf{k}\langle x_1,\dots,x_s\rangle \to \textbf{k}\langle X_1,\dots,X_s\rangle
\end{equation}

We will firstly show that any two commuting elements in the free associative algebra also commute in some algebra of generic matrices.

\begin{lem}
	If we have a commutative subalgebra of transcendent degree 2 in the free associative algebra, then we also have a commutative subalgebra of transcendent degree two if we consider a reduction to generic matrices of a big enough order $n$.
\end{lem}

\begin{proof}
	Assume that $f,g$ are two elements in $A$ generating a commutative subalgebra of transcendent degree two. Suppose $\bar{f},\bar{g}$ are images of homomorphism $\pi$ in $R$ of $f$ and $g$. If $\mathrm{\textit{tr.deg}}\textbf{k}[\bar{f},\bar{g}]=1$, then there exist a polynomial $P_n(\bar{f},\bar{g})=0.$ Recall that in any reduction, if $f,g$ commute in an algebra, then their images $\bar{f},\bar{g}$ in a factor algebra also commute. If $N>n$, then the matrix algebra $M_{n}(\bar{\textbf{k}})$ is a factor of $M_{N}(\bar{\textbf{k}})$, hence $P_N(\bar{f},\bar{g})=0$ in $M_n$.  By Amitsur's theorem (Refer to \cite{Row}, Theorem 3.26, p.176), any algebra of generic matrices is a domain. Then the subalgebra of generic matrices generated by $\bar{f},\bar{g}$ is also a domain. Then $P_n$ divides $P_N$, hence coincides. So there exists a minimal polynomial $P$, such that $P(\bar{f},\bar{g})=0$ for any reduction on $M_n(\bar{\textbf{k}})$.
	
	Suppose $l=\deg P(\bar{f},\bar{g})\leq\deg P\cdot \max\deg(f,g)$, consider reduction to $M_l$, by Amitsur-Levitzki Theorem\cite{Amit50}, $P(f,g)\neq0$ in free associative algebra and its minimal identity of $M_l$ has exactly degree $2l$. Hence $P(f,g)\neq 0$ in $M_l$. Contradiction!
\end{proof}

\section{Quantization}
Let $f,g$ be two commuting generic matrices which embedding to $M_n(\mathbf R)$ with $\mathrm{\textit{tr.deg}}(f,g)=2$, where $n\geq \max\{\deg f,\deg g\}$. Note that $f,g$ are not central polynomials, since $[f,x]$ and $[g,x]$ are not identities because of Amitsur-Levitzki Theorem.

Since the algebra of generic matrix is a domain embedding in a skew field, so the the minimal polynomial for $f$ (resp. g) is an irreducible polynomial. If $f$(resp. $g$) is not central, then eigenvalues of $f$(resp. $g$) are distinct and since the algebra of generic matrices is a domain by Amitsur's theorem, then $f$ could be diagonalizable in some extension $S$ of $\mathbf R$.

This is straightly coming from the result for linear algebra that any $n\times n$ matrix is diagonalizable over the field $\textbf{k}$ if it has n distinct eigenvalues in $\textbf{k}$. Note that if $f$ is diagonalizable, then $g$ is also diagonalizable since they commute.

Consider elements corresponding to $f,g$ in $S[[\mathfrak{h}]]$ which is a $k[[\mathfrak{h}]]$-module of formal power series with coefficients in $S$. Let $*$ be formal deformation\cite{Kell03} of multiplication of $S$, i.e. a $k[[\mathfrak{h}]]$-bilinear map such that we have $$u*v\equiv uv \mod \mathfrak{h} S[[\mathfrak{h}]]$$ for all power series $u,v\in S[[\mathfrak{h}]]$. Hence, the product of two elements $a,b$ of $S$ is then of the form $$a*b=ab+B_{1}(a,b)\mathfrak{h}+\cdots+B_n(a,b)\mathfrak{h}^n+\cdots$$ for a sequence of $k$-bilinear maps $B_i$, after putting $B_0(a,b)=ab$ and we could write $$*=\sum_{n=0}^{\infty}B_n\mathfrak{h}^n.$$
Then $(S,\{,\})$ has free Poisson structure, and there is a natural homomorphism $q: S\to S[[\mathfrak{h}]]$. Note that we have

\begin{equation}
\frac{1}{\mathfrak{h}}([u,v]_*-[u,v])\equiv \{u,v\} \mod \mathfrak{h}
\end{equation}

The notation $[u,v]_*$ is the commutator after quantization which equals to $u*v-v*u$. After introduce quantization, we claim that $f$ (resp. $g$) is also diagonalizable after quantization,

\begin{lem}
	Every generic matrix with distinct eigenvalues after quantization can be diagonalized over some finite extension $S[[\mathfrak{h}]]$ via conjugation modulus $\mathfrak{h}^2$. More precisely, let $A=A_0+\mathfrak{h} A_1$, where $A_0$ has distinct eigenvalues, then $A$ can be diagonalized over some finite extension of  $S[[\mathfrak{h}]]$.
\end{lem}

\begin{proof}
	Without loss of generality, suppose $A_0$ is a diagonal generic matrix with distinct eigenvalues, assume $B=E+\mathfrak{h} T$ and the conjugation inverse $B^{-1}=E- \mathfrak{h}T$, (where $E$ is the identity matrix.) then we have
	$$(E+\mathfrak{h}T)(A_0+\mathfrak{h} A_1)(E-\mathfrak{h}T)=A_0+\mathfrak{h}([T,A_0]+A_1)  \mod \mathfrak{h}^2,$$ then we just need to solve the equation $[T,A_0]=- A_1$.
	
	We suppose that $A_1$ is the matrix with all $0$ diagonal elements. Let $T=(t_{ij}), A_0=\mathrm{diag}\{\lambda_1,\dots,\lambda_n\}$, then $$[T,A_0]=(e_{ij}\lambda_i-\lambda_j e_{ji}).$$ We could determine all elements in $T$ by successive approximation on modulus $\mathfrak{h},\mathfrak{h}^2,\mathfrak{h}^3$ etc.
\end{proof}

Next we will show that two commuting generic matrices $f,g$ with $\mathrm{\textit{tr.deg}} (f,g)=2$ don't commute after quantization.

\begin{lem}
	If $f,g$ are two commuting generic matrices $f, g$ with $\mathrm{\textit{tr.deg}}(f,g)=2$, then after quantization the generic matrices representations of $f,g$ don't commute, i.e. $\frac{1}{\mathfrak{h}}[\hat{f},\hat{g}]_*\neq0 \mod \mathfrak{h}$.
\end{lem}

\begin{proof}
	We have showed that $\hat{f},\hat{g}$ could be diagonalized in $S[[\mathfrak{h}]]$, i.e. we could write them into specific forms modulus $\mathfrak{h}^2$ as follows:
	
	\begin{equation*}
	\hat{f}=\begin{pmatrix}
	\lambda_1+\mathfrak{h} \delta_1 & & 0 \\
	& \ddots & \\
	0 & &\lambda_n+\mathfrak{h}\delta_n\\
	\end{pmatrix}+\mathfrak{h} \begin{pmatrix}
	0 & & * \\
	& \ddots & \\
	* & &0\\
	\end{pmatrix}
	\end{equation*}
	
	\medskip
	
	\begin{equation*}
	\hat{g}=\begin{pmatrix}
	\mu_1+\mathfrak{h} \nu_1 & & 0 \\
	& \ddots & \\
	0 & &\mu_n+\mathfrak{h} \nu_n\\
	\end{pmatrix}+\mathfrak{h} \begin{pmatrix}
	0 & & * \\
	& \ddots & \\
	* & &0\\
	\end{pmatrix}
	\end{equation*}
	
	\medskip
	Then we can get after modulus $\mathfrak{h}$, we have
	
	\begin{align*}
	\frac{1}{\mathfrak{h}}[\hat{f}*\hat{g}-\hat{g}*\hat{f}]&=\begin{pmatrix}
	\frac{1}{\mathfrak{h}}\{\lambda_1,\mu_1\}& & 0 \\
	& \ddots & \\
	0 & &\frac{1}{\mathfrak{h}}\{\lambda_n,\mu_n\}\end{pmatrix}+\vec{\lambda}*\begin{pmatrix}
	0 & & * \\
	& \ddots & \\
	* & &0\\
	\end{pmatrix}\\
	&-\begin{pmatrix}
	0 & & * \\
	& \ddots & \\
	* & &0\\
	\end{pmatrix}*\vec{\lambda} +\begin{pmatrix}
	0 & & * \\
	& \ddots & \\
	* & &0\\
	\end{pmatrix}*\vec{\mu}-\vec{\mu}*\begin{pmatrix}
	0 & & * \\
	& \ddots & \\
	* & &0\\
	\end{pmatrix}\\
	\vspace*{-0.9cm}&+ \mathfrak{h} \left\{\begin{pmatrix}
	0 & & * \\
	& \ddots & \\
	* & &0\\
	\end{pmatrix},\begin{pmatrix}
	0 & & * \\
	& \ddots & \\
	* & &0\\
	\end{pmatrix} \right\}
	\end{align*}
	
	\medskip
	Note that all terms have empty diagonal except the first term, and hence $$\frac{1}{\mathfrak{h}}[\hat{f}*\hat{g}-\hat{g}*\hat{f}]\neq 0 \mod \mathfrak{h}.$$
\end{proof}

\section{Conclusion}

Consider the natural homomorphism $q$ from $R$ to $\hat{R}$, by sending $\cdot\mapsto *$. Then $$q\pi([f,g])=q[\bar{f},\bar{g}]=[\hat{\bar{f}},\hat{\bar{g}}]_*=0$$ which is contradict to the result that $[\hat{\bar{f}},\hat{\bar{g}}]_*\neq0$ which we just have showed above. Hence, we finish the proof of our main theorem.

Our main theorem directly implies Bergman's theorem on centralizers. Since any centralizer of $A$ is a commutative subalgebra(refer to \cite{Berg69}, Proposition 2.2, p331), hence of transcendental degree 1. So it is a commutative subalgebra with form $\textbf{k}[z]$ for some $z\in A$.

\section*{Acknowledgement}
This work was supported by the Russian Science Foundation (RNF) educational scholarship with number (\textit{No} 17-11-01377) for financial support. 


\begin{thebibliography}{99}
	\bibitem{Berg69} Bergman G M. Centralizers in free associative algebras[J]. Transactions of the American Mathematical Society, 1969, 137: 327-344.
	\addcontentsline{toc}{section}{References} 
	
	\bibitem{Belov05} Belov-Kanel A, Kontsevich M. Automorphisms of the Weyl algebra[J]. Letters in mathematical physics, 2005, 74(2): 181-199.
	
	\bibitem{Belov07} Belov-Kanel A, Kontsevich M. The Jacobian conjecture is stably equivalent to the Dixmier conjecture[J]. Mosc. Math. J, 2007, 7(2): 209-218.
	
	\bibitem{Kole00} Kolesnikov P S. The Makar-Limanov algebraically closed skew field[J]. Algebra and Logic, 2000, 39(6): 378-395.
	
	\bibitem{Kole01} Kolesnikov P S. Different definitions of algebraically closed skew fields[J]. Algebra and Logic, 2001, 40(4): 219-230.
	
	\bibitem{Art78} Artamonov V. A. Projective metabelian groups and Lie algebras. (Russian) Izv. Akad. Nauk SSSR Ser. Mat. 42 (1978), no. 2, 226?236, 469
	
	\bibitem{Art84} Artamonov V. A. Projective modules over universal enveloping algebras. (Russian) Izv. Akad. Nauk SSSR Ser. Mat. 48 (1984), no. 6, 1123?1137.
	
	\bibitem{Art91} Artamonov V. A. Nilpotency, projectivity and decomposability. (Russian) Sibirsk. Mat. Zh. 32 (1991), no. 6, 3--11, 203; translation in Siberian Math. J. 32 (1991), no. 6, 901?909 (1992)
	
	\bibitem{Art98} Artamonov V. A. The quantum Serre problem. (Russian) Uspekhi Mat. Nauk 53  (1998), no. 4(322), 3--76; translation in Russian Math. Surveys 53 (1998), no. 4, 657?730
	
	\bibitem{Art08} Artamonov V A. Quantum polynomials Advances in algebra and combinatorics, vol. 19-34[J]. 2008.
	
	\bibitem{Umir12} Umirbaev U. Universal enveloping algebras and universal derivations of Poisson algebras[J]. Journal of Algebra, 2012, 354(1): 77-94.
	
	\bibitem{Umir07} Umirbaev U. The Anick automorphism of free associative algebras[J]. Journal für die reine und angewandte Mathematik (Crelles Journal), 2007, 2007(605): 165-178.
	
	\bibitem{Kell03} Keller B. Notes for an Introduction to Kontsevich's quantization theorem[J]. 2003.
	
	\bibitem{Artin99} Artin M. Noncommutative Rings. Class Notes, fall 1999:75-77.
	
	\bibitem{Amit50} Amitsur A S. Levitzki J. Minimal identities for algebras[J]. Proceedings of the American Mathematical Society, 1950, 1(4): 449-463.
	
	\bibitem{Row} Rowen L. Polynomial identities in ring theory, Academic press, New York (1980).

\end{thebibliography}
\end{document}